\documentclass[a4, 12pt, reqno]{amsart}
\usepackage{amsmath,amsthm,amsfonts,amssymb,amscd, amsxtra}
\usepackage{graphicx}
\usepackage{mathrsfs}
\usepackage{mathtools}
\usepackage{etoolbox}
\usepackage{scalerel,stackengine}
\usepackage{blkarray}
\usepackage{hhline}
\usepackage{centernot,tikz}
\DeclareMathOperator{\spn}{span}
\newtheorem{theorem}{Theorem}[section]
\newtheorem{lemma}[theorem]{Lemma}

\newtheorem{prob}[theorem]{Problem}
\usepackage{xcolor}
\newtheorem{definition}[theorem]{Definition}
\newtheorem{example}[theorem]{Example}

\newtheorem{prop}[theorem]{Proposition}
\newtheorem{corollary}[theorem]{Corollary}
\newtheorem{remark}[theorem]{Remark}

\numberwithin{equation}{section}
\begin{document}
\title[Representation and normality  of Hyponormal operators in $\overline{\mathcal{AN}(H)}$ ]{Representation and normality of Hyponormal operators in the closure of $\mathcal{AN}$-operators}
\author{G. Ramesh}
 \address{G. Ramesh, Department of Mathematics, IIT Hyderabad, Kandi, Sangareddy, Telangana- 502284, India.}
 \email{rameshg@math.iith.ac.in}
\author{Shanola S. Sequeira}
 \address{Shanola S. Sequeira, Department of Mathematics, IIT Hyderabad, Kandi, Sangareddy, Telangana- 502284, India.}
 \email{ma18resch11001@iith.ac.in}
\maketitle
\begin{abstract} Let $H_1$, $H_2$ be complex Hilbert spaces. A bounded linear operator $T : H_1 \to H_2$ is said to be norm attaining if there exists a unit vector $x \in H_1$ such that $\|Tx\| = \|T\|$. If $T|_{M} : M \to H_2$ is norm attaining for every closed subspace $M$ of $H_1$, then we say that $T$ is an absolutely norm attaining ($\mathcal{AN}$-operator). If the norm of the operator is replaced by the minimum modulus $m(T) = \inf\{\|Tx\| : x \in H_1, \|x\| =1\}$, then $T$ is said to be a minimum attaining and an absolutely minimum attaining operator ($\mathcal{AM}$-operator), respectively.

	In this article, we give representations of quasinormal $\mathcal{AN}$, $\mathcal{AM}$-operators and the operators in the closure of these two classes.  Later we extend these results to  the class of hyponormal operators in the closure of $\mathcal{AN}$-operators and further look at some sufficient conditions under which these operators become normal.
\end{abstract}

\section{Introduction}

In the literature, a vast study has been done on the class of non-normal operators and on the conditions implying the normality of them. The class of non-normal operators includes quasinormal, hyponormal, paranormal operators, etc. and the conditions include compactness of operators, area of the spectrum i.e., operators having spectrum with area measure zero, countability of the spectrum, etc. We refer to \cite{ANDO,BER1,Putnam,QIU,STA} for more details on these topics.

One of the important classes of non-normal operators which invoked interest in a lot of researchers is the class of hyponormal operators \cite{BER1,Putnam,STA}. This class contains normal operators but yet smaller in $\mathcal{B}(H)$, the space of all bounded linear operators on a Hilbert space $H$. More details on hyponormal operators can be found in \cite{MAR,XIA}. Much study is done on the conditions under which a hyponormal operator becomes normal. A few developments in this regard are the compactness of hyponormal operators, which implies normality \cite{ANDO,STA} and  the Putnam's inequality \cite[Theorem 4.1, Page 31]{MAR} which infers that if the spectrum of a hyponormal operator has area measure zero, then it has to be normal \cite{STA}. Several other sufficient conditions are also studied under which a hyponormal operator becomes normal. We refer \cite{MAR} for more details.

We can replace compactness property by a weaker property called $\mathcal{AN}$-property or the absolutely norm attaining property of operators. This class contains compact operators, the partial isometries with finite dimensional nullspace and was first introduced by Carvajal and Neves in \cite{CAR1}. We refer to \cite{PAN,RAM2,VEN} for more details and the spectral theory of this class of operators. In \cite{BALA2}, it is shown that if the essential spectrum and the Weyl's spectrum of hyponormal $\mathcal{AN}$-operators are the same, then such operators become normal. In \cite{NB *-para}, it is also shown that invertible hyponormal $\mathcal{AN}$-operators are normal.

The present article generalizes the results of \cite{BALA1,BALA2,RAM2}. We are mainly interested to look at the conditions for normality of hyponormal operators when it belongs to the closure of $\mathcal{AN}$-operators, as this class contains $\mathcal{AN}$-operators and not contained in the set of norm attaining operators. Here the closure is taken with respect to the operator norm topology. See \cite{RAMSSS} for a detailed study of  the closure of $\mathcal{AN}$-operators.

 Analogous to absolutely norm attaining operators, a class of absolutely minimum attaining operators ($\mathcal{AM}$-operators) was also studied by Carvajal and Neves in \cite{CAR2} and further developments of this class can be found in \cite{BALA1,GAN}.
 
Yoshino in \cite{YOS} showed that the only paranormal norm attaining Toeplitz operator is a scalar multiple of an isometry and paranormal norm attaining Hankel operator is normal. In a recent article \cite{RAMSSSJMAA}, the authors studied a characterization of $\mathcal{AN}$-Toeplitz and $\mathcal{AM}$-Hankel operators.

In \cite{RAMSSS}, it was shown that the closure of $\mathcal{AN}$-operators is the same as that of the closure of $\mathcal{AM}$-operators. It can be seen that this class contains the class of compact operators, isometries, partial isometries with finite dimensional nullspace etc., hence not equal to $\mathcal{B}(H)$, where $H$ is infinite dimensional. The characterization of positive operators in $\overline{\mathcal{AN}(H)}$ and a representation of normal operators is also discussed in the same paper. Hence in the present article we continue our discussion on a few non-normal operators in $\overline{\mathcal{AN}(H)}$ especially hypornormal operators and a sub-class of it called quasinormal operators. This paper mainly focuses on the following results.
\begin{enumerate}
	\item A representation of quasinormal $\mathcal{AN}$ and $\mathcal{AM}$-operators
	\item A representation of quasinormal operators in the closure of $\mathcal{AN}$-operators.
	\item A representation of hyponormal $\mathcal{AM}$-operators. \item A representation of hyponormal operators in $\overline{\mathcal{AN}(H)}$.
 \item Conditions implying the normality of hyponormal operators in $\overline{\mathcal{AN}(H)}$, such as, invertibility of operators, operators with the same essential and the Weyl's spectrum etc.
\end{enumerate}

We organize this paper as follows: In the remaining part of this section we recall some preliminaries which are needed for developing the article. In section 2, we prove some basic properties satisfied by the operators in $\overline{\mathcal{AN}(H)}$. In section 3, we give a representation  of  quasinormal $\mathcal{AN}$, $\mathcal{AM}$-operators and operators in $\overline{\mathcal{AN}(H)}$. In section 4, we give a representation of hyponormal operator in $\overline{\mathcal{AN}(H)}$, as a consequence we get a representation of operators in $\overline{\mathcal{AN}(H)}$, when its adjoint is hyponormal. Finally in the last section, we give some sufficient conditions under which these operators become normal.
\subsection{Preliminaries}
Throughout the article let $H, H_1, H_2$ denote infinite dimensional complex Hilbert spaces and $\mathcal{B}(H_1, H_2)$ denote the Banach space of all bounded linear operators from $H_1$ into $H_2$. Let $C(0, \alpha)$ denote the sphere with center zero and radius $\alpha$.
We denote the range and nullspaces of $T \in \mathcal{B}(H_1, H_2)$ by $R(T)$ and $N(T)$, respectively. If $R(T)$ is finite dimensional, then $T$ is said to be finite rank operator and $T$ is compact if the image of a bounded set has compact closure. We denote the set of all finite rank and compact operators from $H_1$ to $H_2$ by $\mathcal{F}(H_1, H_2)$ and $\mathcal{K}(H_1, H_2)$, respectively.  If $H_1 = H_2 = H$, then we write $\mathcal{B}(H_1, H_2) = \mathcal{B}(H)$, $\mathcal{K}(H_1, H_2) = \mathcal{K}(H)$ and $\mathcal{F}(H_1, H_2) = \mathcal{F}(H)$.

Given any $T \in \mathcal{B}(H_1, H_2)$, the adjoint operator $T^* \in \mathcal{B}(H_2, H_1)$ satisfies the following condition.

\[\langle Tx, y \rangle = \langle x, T^{*}y \rangle, \ \forall x \in H_1,\  y \in H_2.\]

Next we define the notion of the spectrum of an operator. If $T \in \mathcal{B}(H)$, then we define spectrum of $T$ by
\[\sigma(T) := \{\lambda \in \mathbb{C} : T - \lambda I \ \text{is} \ \text{not} \ \text{invertible} \ \text{in} \ \mathcal{B}(H)\}.\]

It is well known that the spectrum of $T$ can be decomposed as the disjoint union of the point spectrum $\sigma_{p}(T):= \{\lambda \in \mathbb{C} : T-\lambda I \ \text{is}\; \text{not}\; \text{one-one} \ \text{in} \ \mathcal{B}(H)\}$, the residual spectrum $\sigma_{r}(T) := \{\lambda \in \mathbb{C} : T-\lambda I \ \text{is}\; \text{one-one}\; \text{but}\; \overline{R(T-\lambda I)} \neq H\}$ and the continuous spectrum $\sigma_{c}(T) := \sigma(T) \setminus (\sigma_{p}(T) \cup \sigma_{r}(T))$.

 If $T \in \mathcal{B}(H_1, H_2)$, we call $|T| = (T^*T)^{1/2}$ as the modulus of $T$ and there exists a unique partial isometry $W \in \mathcal{B}(H_1, H_2)$ such that $T = W |T|$ satisfying $N(W) = N(T)$. This is called the polar decomposition of $T$.

 We say $T \in \mathcal{B}(H)$ is normal if $TT^* = T^*T$, self-adjoint if $T = T^*$ and $T$ is positive if $\langle Tx, x\rangle \geq 0$. If $\mathcal{A} \subset \mathcal{B}(H)$, then the set of all positive operators in $\mathcal{A}$ is denoted by $\mathcal{A}_{+}$. For self-adjoint operators $A$ and $B$ if $B-A \in \mathcal{B}(H)_{+}$, then we denote this by $A \leq B$. For more details about the basic definitions and results in operator theory we refer to \cite{CON,Gohberg,FUR,KUB,ReedSimon}.

	Let $T \in \mathcal{B}(H)$ and $\pi : \mathcal{B}(H) \to \mathcal{B}(H)/ \mathcal{K}(H)$ be the quotient map. Then $T$ is said to be a Fredholm operator if $\pi(T)$ is invertible in $\mathcal{B}(H)/ \mathcal{K}(H)$.

We can use the following result as an equivalent definition for the Fredholm operator.
\begin{prop}\cite[Proposition 2.10, Page 359]{CON} \label{Fredholmdef}
	An operator $T \in \mathcal{B}(H)$ is a Fredholm operator if and only if $R(T)$ is closed and both $N(T)$ and $N(T^*)$ are finite dimensional.
\end{prop}
The index of a Fredholm operator $T$ is defined by
\[ ind(T) = \text{dim} N(T) - \text{dim} N(T^*). \]
Writing $\mathscr{F}$ as the class of all Fredholm operators and $\mathscr{F}_0$ as the class of all Fredholm operators with index zero, we define the essential spectrum of $T$ by
\[\sigma_{ess}(T) = \{\lambda \in \mathbb{C} : T -\lambda I \notin \mathscr{F}\}.\]

The above definition of the essential spectrum can be found in \cite[Propositions 4.2, 4.3]{CON}. The essential spectrum of $T$ can also be defined by
$\sigma_{ess}(T) = \sigma(\pi(T))$.

Similarly the Weyl spectrum of $T$ \cite{COB} is defined by
\[\omega(T) = \{\lambda \in \mathbb{C} : T - \lambda I  \notin \mathscr{F}_0\}.\]

More details on the essential spectrum can be found in \cite{CON,MUL,ReedSimon} and that of the Weyl spectrum can be found in \cite{BER2,COB}.

Another useful definition of an essential spectral point of a self-adjoint operator is given in the following theorem.
\begin{theorem}\label{ess spectrum of self-adjoint}\cite[Theorem VII.11, Page 236]{ReedSimon}
	Let $T = T^* \in \mathcal{B}(H)$. Then $\lambda \in \sigma_{ess}(T)$ if and only one or more of the following conditions hold.
	\begin{enumerate}
		\item $\lambda$ is an eigenvalue of infinite multiplicity.
		\item $\lambda$ is the limit point of $\sigma_{p}(T)$.
		\item $\lambda \in \sigma_{c}(T)$.
	\end{enumerate}
\end{theorem}

For a self-adjoint operator $T$, the set $\sigma_{d}(T) = \sigma(T)\setminus \sigma_{ess}(T)$ is called the discrete spectrum of $T$. More precisely, $\sigma_{d}(T) = \{\lambda \in \sigma(T): \lambda \  \text{is} \ \text{an} \ \text{isolated} \ \text{eigenvalue}\ \text{with} \ \text{finite} \ \text{multiplicity} \}$.

For $T \in \mathcal{B}(H)$, let $\pi_{00}(T)$ denote the set of all isolated eigenvalues of $T$ with finite multiplicity. In particular if $T = T^*$, we have $\pi_{00}(T) = \sigma_{d}(T)$.

\begin{remark}
	There are various definitions for the essential spectrum of an operator, but all the definitions coincide when the operator is self-adjoint. See \cite[Section 3.14]{BarrySimon4} for more details.
\end{remark}

For $T \in \mathcal{B}(H)$, the essential minimum modulus of $T$ \cite{BOU} is defined by
\[ m_e(T) = \inf \{\lambda : \lambda \in \sigma_{ess}(|T|)\}.\]

Next we define two classes of non-normal operators.

\begin{definition}
	Let $T \in \mathcal{B}(H)$, then
	\begin{enumerate}
		\item $T$ is quasinormal if $T$ commutes with $T^*T$.
		\item $T$ is hyponormal if $T^*T - TT^* \geq 0$.
	\end{enumerate}
\end{definition}
Clearly we can see that these classes contain the class of normal operators and also quasinormal operators are subsets of hyponormal operators. (See \cite{FUR} for more details).

\begin{definition}\cite[Page 30]{CON}
	Let $\{H_n\}_{n \in \mathbb{N}}$ be the family of Hilbert Spaces and  $T_n \in \mathcal{B}(H_n)$ for all $n \in \mathbb{N}$.  Let $H  = \displaystyle{\bigoplus_{n \in \mathbb{N}}} H_n$. Then we define direct sum of $T = \displaystyle{\bigoplus_{n \in \mathbb{N}}} T_n : H \to H$ of $(T_n)$ by
	\begin{equation*}
	T(x_1, x_2,\dots) = (T_1x_1, T_2x_2,\dots), \ \forall \ x_n \in H_n, n \in \mathbb{N} .
	\end{equation*}
	If $\displaystyle{\sup_{n \in \mathbb{N}}} \{\|T_n\|\} < \infty$, then $T \in \mathcal{B}(H)$  and $\|T\| = \displaystyle{\sup_{n \in \mathbb{N}}}\{\|T_n\|\}$.
\end{definition}

\begin{definition} \cite[Definition 1.1, 1.2]{CAR1}
	A bounded operator $T: H_1 \to H_2$ is called norm attaining if there exists $x \in H_1$  with $\|x\|= 1$ such that $\|Tx\| = \|T\|$ and absolutely norm attaining ($\mathcal{AN}$-operator) if $T|_{M} : M \to H_{2}$ is norm attaining for every closed subspace $M$ of $H_1$.
\end{definition}

Similar to the concept of norm of an operator, we have the concept of minimum modulus of the operator which is denoted as $m(T)$ and defined by
\begin{equation*}
m(T):= \inf\{\|Tx\| : x \in H_1, \|x\| = 1 \}.
\end{equation*}
\begin{definition}\cite[Definition 1.1, 1.4]{CAR2}
	An operator $T \in \mathcal{B}(H_1, H_2)$ is called  minimum attaining if there exists $x_0 \in S_{H_1}$ such that $m(T) = \|Tx_0\|$. If $T|_M : M \to H_2$  is minimum attaining for every closed subspace $M$ of $H_1$,  then $T$ is called absolutely minimum attaining or $\mathcal{AM}$- operator.
\end{definition}
The sets of all $\mathcal{AN}$ and $\mathcal{AM}$-operators from $H_1$ to $H_2$ are denoted by $\mathcal{AN}(H_1,H_2)$ and $\mathcal{AM}(H_1,H_2)$, respectively and when $H_1=H_2=H$, we write $\mathcal{AN}(H,H) := \mathcal{AN}(H)$ and $\mathcal{AM}(H,H) := \mathcal{AM}(H)$. More details on $\mathcal{AN}$ operators can be found in \cite{CAR1,PAN,RAM2,VEN,BALA2} and that of  $\mathcal{AM}$ operators can be found in \cite{BALA1,CAR2,GAN}.

Next we quote one of the important results proved in \cite{RAMSSS} regarding the characterization for positive operators in $\overline{\mathcal{AN}(H)}$ in terms of the essential spectrum.
\begin{theorem}\cite [Theorem 4.5]{RAMSSS}\label{positiveessentialspectrum}
		Let $T \in \mathcal{B}(H)_{+}$. Then $T \in \overline{\mathcal{AN}(H)}_{+}$  if and only if $\sigma_{ess}(T)$ is a singleton set.
\end{theorem}

Next we state the result which mainly invoked the question for developing this article.

\begin{theorem} \label{compacthypo} \cite{ANDO,BER1,STA}
A compact hyponormal operator is normal.
\end{theorem}

\section{Some results of operators in $\overline{\mathcal{AN}(H)}$}
We start this section by stating the following characterization of positive operators in $\overline{\mathcal{AN}(H)}$.
\begin{theorem}\cite[Theorem 4.2]{RAMSSS}\label{positivechar}
	Let $T \in \mathcal{B}(H)$. Then the following are equivalent.
	\begin{enumerate}
		\item $ T \in \overline{\mathcal{AN}(H)}_{+}$.
		\item There exists $\alpha \geq 0$, $K_1, K_2 \in \mathcal{K}(H)_{+}$ with $K_1K_2 = 0$ and $K_1 \leq \alpha I$ such that $T = \alpha I - K_1 + K_2$. Moreover this representation is unique.
	\end{enumerate}
\end{theorem}

 From the structure of positive operators in $\overline{\mathcal{AN}(H)}_{+}$, we make a few observations.
 These observations are the generalizations of the results proved in \cite{VEN}.

\begin{prop}
	Let $T = \alpha I -K_1 +K_2$, where $K_1, K_2 \in \mathcal{K}(H)_{+}$ with $K_1K_2 = 0$ and $K_1 \leq \alpha I$. If $T \notin \mathcal{K}(H)$, then the following holds true.
	\begin{enumerate}
	\item\label{nullspacesubsetnullspace} $N(T) \subseteq N(K_2)$.
	\item\label{nullspacesubsetrange} If $x \in N(T)$, then $K_1x = \alpha x$. Hence $N(T) \subseteq R(K_1)$. Moreover in this case $\|K_1\| = \alpha$.
	\item \label{Tinjective}$T$ is not injective iff $\|K_1\| = \alpha$.
	\item \label{TFredholm}$T$ is Fredholm and $m_e(T)= \{\alpha\}$.
\end{enumerate}
\end{prop}

\begin{proof}
Clearly $\alpha \neq 0$ as $T \notin \mathcal{K}(H)$.

	Proof of (\ref{nullspacesubsetnullspace}). If $x \in N(T)$, then $Tx = (\alpha I - K_1 + K_2)x = 0$. This implies $K_2x = - \alpha x + K_1x$. Since $K_1K_2 =0$, we get $K^2_2x = -\alpha K_2x$. If $K_2x \neq 0$, then $-\alpha \in \sigma_p(K_2)$, a contradiction to $K_2 \in \mathcal{K}(H)_+$. Hence $x \in N(K_2)$.
	
	Proof of (\ref{nullspacesubsetrange}). Let $x \in N(T)$. Then from (\ref{nullspacesubsetnullspace}), $K_2 x = 0$. Hence $K_1x = \alpha x$. This implies $N(T) \subseteq R(K_1)$ and $\|K_1\| = \alpha$.
	
	Proof of (\ref{Tinjective}). If $T$ is not injective, then from (\ref{nullspacesubsetrange}), $\|K_1\| = \alpha$.
	
	Conversely, let $\|K_1\| = \alpha$. Since $K_1 \in \mathcal{K}(H)$ is norm attaining, hence there exists $x \neq 0$ such that $K_1x = \alpha x$ \cite[Proposition 2.4]{CAR1}. As $K_2K_1 = 0$, we get
	\begin{equation}
	T(K_1x) = K_1(\alpha x - K_1x) + K_2 K_1 x = 0.
	\end{equation}
	Therefore $T$ is not injective.
	
	Proof of (\ref{TFredholm}) As $T \in \overline{\mathcal{AN}(H)}_+$, by \cite[Corollary 3.8]{RAMSSS}, $N(T)$ is finite dimensional and $R(T)$ is closed. As $T = T^*$, even $N(T^*)$ is finite dimensional. Therefore by Proposition \ref{Fredholmdef}, $T$ is Fredholm. Now, by the Weyl's theorem for the essential spectrum, we have $\sigma_{ess}(T)  = \{\alpha\}$. Hence $m_e(T) = \alpha$.	
\end{proof}
\begin{prop} \label{prop1}
	Let $T = \alpha I - K_1 + K_2$, where $\alpha \geq 0$, $K_1, K_2 \in \mathcal{K}(H)_{+}, K_1 \leq \alpha I$ and $K_1K_2 = 0$. Then the following result holds true.
\begin{enumerate}
	\item \label{rangereduces} $\overline{R(K_1)}$ reduces $T$.
	\item \label{positiveTmatrix}
	\[T =
	\begin{blockarray}{ccc}
	N(K_1) & N(K_1)^\bot\\
\begin{block}{(cc)c}
\alpha I_{N(K_1)} +	\widetilde{K_2} & 0 & N(K_1) \\
{}\\
	0 & \alpha I_{N(K_1)^{\bot}} -	\widetilde{K_1} & N(K_1)^\bot\\
	\end{block}
	\end{blockarray},\]
	where $\widetilde{K_2} = K_2|_{N(K_1)}$ and $\widetilde{K_1} = K_1|_{N(K_1)^{\bot}}$.
\end{enumerate}
\end{prop}
\begin{proof}
	Proof of (\ref{rangereduces}). Since $T$ is positive, it is enough to prove that $R(K_1)$ is invariant under $T$.
	
	If $y \in R(K_1)$, then $y = K_1x$ for some $x \in H$. So
	\begin{equation*}
	Ty = (\alpha I - K_1 + K_2)K_1 x = K_1(\alpha I - K_1)x \in R(K_1).
	\end{equation*}
	Hence $\overline{R(K_1)}$ is invariant under $T$.

Proof of (\ref{positiveTmatrix}). Let $H = N(K_1) \oplus {N(K_1)}^{\bot}$. As $K_1K_2 =0$, we get $\overline{R(K_2)} \subseteq N(K_1)$. Therefore $N(K_1)$ is invariant under $K_2$. Since $ K_2 \in \mathcal{K}(H)_{+}$, $N(K_1)$ reduces $K_2$.

Also ${N(K_1)}^\bot = \overline{R(K_1)}$ reduces $K_1$ by (\ref{rangereduces}). Hence we get

\[K_1 =
\begin{blockarray}{ccc}
N(K_1) & N(K_1)^\bot\\
\begin{block}{(cc)c}
0 & 0 & N(K_1)\\
0 & \widetilde{K_1} & N(K_1)^\bot \\
\end{block}
\end{blockarray} \text{and }
K_2 =
\begin{blockarray}{ccc}
N(K_1) & N(K_1)^\bot\\
\begin{block}{(cc)c}
 \widetilde{K_2} & 0 & N(K_1)\\
0 & 0 & N(K_1)^\bot \\
\end{block}
\end{blockarray}.\] Therefore,
\begin{equation*}
\begin{split}
T & = \alpha I -K_1 +K_2\\ & =\begin{blockarray}{cc}
\begin{block}{(cc)}
\alpha I_{N(K_1)} +	\widetilde{K_2} & 0 \\
{}\\
0 & \alpha I_{N(K_1)^{\bot}} -	\widetilde{K_1} \\
\end{block}
\end{blockarray} \ .
\end{split}
\end{equation*}
 Note that if $\overline{R(K_1)} = H$, then $T = \alpha I -\widetilde{K_1}$.
\end{proof}

In \cite{VEN}, it is proved that if $T \in \mathcal{AN}(H)$ is invertible, then $ T^{-1} \in \mathcal{AM}(H)$. Next we prove that the inverse of an operator in $\overline{\mathcal{AN}(H)}$, also belongs to $\overline{\mathcal{AN}(H)}$.

\begin{theorem}
	If $T \in \overline{\mathcal{AN}(H)}$ be invertible, then $T^{-1} \in \overline{\mathcal{AN}(H)}$.
\end{theorem}
\begin{proof}
	Let $T= U|T|$ be the polar decomposition of $T$. As $T$ is invertible, we have $U$ is unitary (\cite[Corollary 5.91, Page 405]{KUB}) and $|T|$ is invertible.
	
 By \cite[Lemma 3.13]{RAMSSS}, we have $|T| \in \overline{\mathcal{AN}(H)}_{+}$. Hence from Theorem \ref{positivechar}, we get $|T| = \alpha I - K_1 + K_2$, where $\alpha \geq 0, K_1, K_2 \in \mathcal{K}(H)_{+}$ with $K_1K_2 = 0$ and $K_1 \leq \alpha I$. Clearly $\alpha \neq 0$, as $|T|$ is invertible.

 Now,
 \begin{align*}
|T|^{-1}& = {\alpha}^{-1}I - \big({\alpha}^{-1}I - (\alpha I - K_1 + K_2)^{-1} \big),\\&= {\alpha}^{-1}I - \big({\alpha}^{-1}(K_2 -K_1)(\alpha I - K_1 + K_2)^{-1} \big), \\&= {\alpha}^{-1}I +K_3,
 \end{align*}
 where $K_3 = {\alpha}^{-1}(K_1 -K_2)(\alpha I - K_1 + K_2)^{-1} \in \mathcal{K}(H)$. Hence by \cite[Proposition 3.5]{RAMSSS}, $|T|^{-1} \in \overline{\mathcal{AN}(H)}_{+}$.

As $U^*$ unitary, $U^* \in \overline{\mathcal{AN}(H)}$. Therefore by \cite[Theorem 3.11]{RAMSSS}, we conclude that $T^{-1} = |T|^{-1}U^{*} \in \overline{\mathcal{AN}(H)}$.
\end{proof}

\begin{remark}
	If $T \in \overline{\mathcal{AN}(H)}$ is self-adjoint and $\lambda$ is a purely imaginary number, then $T \pm \lambda I \in \overline{\mathcal{AN}(H)}$.
	\begin{proof}
		As $T \in \overline{\mathcal{AN}(H)}$ is self-adjoint, there exists a sequence of self-adjoint operators $\{T_n\} \subseteq \mathcal{AN}(H)$ such that $T_n \to T$ in the operator norm. So, $T_n \pm \lambda I \to T \pm \lambda I$ as $n \to \infty$ and by \cite[Theorem 2.12]{VEN}, we have $T_n \pm \lambda I \in \mathcal{AN}(H)$. Hence $T \pm \lambda I \in \overline{\mathcal{AN}(H)}$.
	\end{proof}
\end{remark}

\section{Quasinormal operators in $\overline{\mathcal{AN}(H)}$}

In this section, we first look at a representation of absolutely norm attaining quasinormal operators by generalizing the results of normal $\mathcal{AN}$-operators in \cite{RAM2}.

\begin{lemma}\cite[Proposition 3.8]{RAM4}\label{quasinormalreducingsubspace}
	Let $T \in \mathcal{B}(H)$ be quasinormal and $T = W|T|$ be the polar decomposition of $T$. If $\lambda \in \sigma_{p}(|T|)$, then $N(|T|- \lambda I)$ reduces $W$ and hence reduces $T$.
\end{lemma}

\begin{theorem}\label{ANquasinormal}
	Let $T \in \mathcal{AN}(H)$ be quasinormal. Then there exists reducing subspaces $H_\infty, H_1$ and $H_2$ with $\text{dim}(H_2) < \infty$ such that $H = H_1 \oplus  H_\infty  \oplus H_2$ and with respect to this decomposition $T$ can be written as, \begin{equation}
	T = \displaystyle \bigoplus_{i = 1}^{n} \alpha_i U_i \bigoplus \alpha V_\infty \bigoplus_{j = 1}^{m} \beta_jV_j,
	\end{equation} for some $n \in \mathbb{N} \cup \{\infty\}$, $m \in \mathbb{N}$, where
	\begin{enumerate}
		\item [(i)]  $\sigma_{ess}(|T|) = \{\alpha\}, \alpha \geq 0$, $H_\infty = N(|T| - \alpha I)$  and $V_\infty \in \mathcal{B}(H_\infty)$ is an isometry.
		
		\item [(ii)] $\{\alpha_i\}_{i =1}^{n} = (\alpha, \|T\|] \cap \sigma(|T|)$, $H_1 = \displaystyle \bigoplus_{i = 1}^{n} N(|T| - \alpha_i I)$ and $U_i \in \mathcal{B}(N(|T| - \alpha_iI))$ is a unitary operator for all $i =1,2, \dots, n$, $n \in \mathbb{N} \cup \{\infty\}$.
		\item [(iii)] $ \{\beta_j\}_{j =1}^{m} = [m(T), \alpha) \cap \sigma(|T|)$, $H_2 = \displaystyle \bigoplus_{j = 1}^{m} N(|T| - \beta_j I)$ and $V_j \in \mathcal{B}(N(|T| - \beta_jI))$ is a unitary operator for all $j =1,2, \dots, m$, $m \in \mathbb{N}$.
\end{enumerate}
\end{theorem}\begin{proof}
	Let $T = W|T|$ be the polar decomposition of $T$. Since $|T| \in \mathcal{AN}(H)$, we have $\sigma_{ess}(|T|)$ is a singleton set say, $\{\alpha\}$ and $[m(T), \alpha) \cap \sigma(|T|)$ is a finite set \cite[Theorem 2.4]{RAM2}.
	
	If $\alpha = 0$, then $T$ is compact and it is well known that, every quasinormal operator is hyponormal \cite[Theorem 1, Page 104]{FUR}. Therefore by Theorem \ref{compacthypo}, we get that $T$ is normal. Hence the required structure of $T$ is obtained by \cite[Theorem 3.9]{RAM2}.
	
	So we assume that $\alpha \neq 0$. Let  $[m(T), \alpha) \cap \sigma(|T|) = \{\beta_j\}_{j=1}^{m}, m \in \mathbb{N}$ and $(\alpha, \|T\|] \cap \sigma(|T|) = \{\alpha_i\}_{i=1}^{n}, n \in \mathbb{N} \cup \{\infty\}$.
	
	
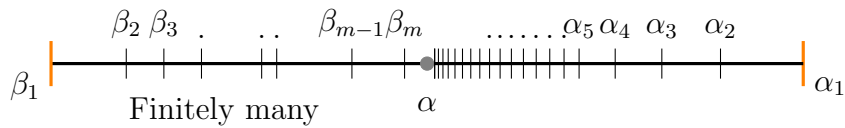
\hspace{1cm}
	\begin{tikzpicture}
	\draw[orange, very thick](5,-.3)--(5,.3);
	\filldraw(5,-.3)  ;
	\draw[very thick]  (-5,0)--(5,0);
	\filldraw [gray] (0,0) circle (2.5pt);
	\filldraw(0,-.3)  node[anchor=north] {$\alpha$};
	\draw [orange, very thick](-5,-.3)--(-5,.3);
	\filldraw(-5,-.3) node[anchor=east] {$\beta_1$};
	\draw (-4,-.2)--(-4,.2) node[anchor=south] {$\beta_2$};
	\filldraw(-2.7,-.3)  node[anchor=north] {Finitely many};
	\draw(5,-.3)  node[anchor=west] {$\alpha_1$};
	\draw (-2.2,-.2)--(-2.2,.2) node[anchor=south] {.};
	\draw (-3,-.2)--(-3,.2) node[anchor=south] {.};
	\draw (-3.5,-.2)--(-3.5,.2) node[anchor=south] {$\beta_3$};
	\draw (-2,-.2)--(-2,.2)node[anchor=south] {.};
	\draw (-1,-.2)--(-1,.2) node[anchor=south] {$\beta_{{m-1}}$};
	\draw (-.3,-.2)--(-.3,.2) node[anchor=south] {$\beta_{m}$};
	\draw (.1,-.2)--(.1,.2) ;
	
	\draw (.10,-.2)--(.10,.2);
	\draw (.15,-.2)--(.15,.2);
	\draw (.21,-.2)--(.21,.2);
	\draw (.28,-.2)--(.28,.2);
	\draw (.37,-.2)--(.37,.2);
	\draw (.47,-.2)--(.47,.2);	
	\draw (.58,-.2)--(.58,.2);
	\draw (.7,-.2)--(.7,.2);
	\draw (.83,-.2)--(.83,.2) node[anchor=south] {$.$};
	\draw (.97,-.2)--(.97,.2) node[anchor=south] {$.$};
	\draw (1.12,-.2)--(1.12,.2)node[anchor=south] {$.$};	
	\draw (1.28,-.2)--(1.28,.2) node[anchor=south] {$.$};
	\draw (1.45,-.2)--(1.45,.2) node[anchor=south] {$.$};
	\draw (1.63,-.2)--(1.63,.2) node[anchor=south] {$.$};
	\draw (1.82,-.2)--(1.82,.2) node[anchor=south] {$.$};
	\draw (2.02,-.2)--(2.02,.2) node[anchor=south] {$\alpha_5$};
	\draw (2.5,-.2)--(2.5,.2) node[anchor=south] {$\alpha_4$};
	\draw (3.12,-.2)--(3.12,.2) node[anchor=south] {$\alpha_3$};
	\draw (3.9,-.2)--(3.9,.2) node[anchor=south] {$\alpha_2$};
	\end{tikzpicture}
	
	\hspace{4cm} Spectral Diagram of $|T| \in \mathcal{AN}(H)$
	
	\vspace{0.5cm}Define $M_i := N(|T| - \alpha_i I),\ i =1,2,\dots,n$ for $n \in \mathbb{N} \cup \{\infty\}$, $H_\infty := N(|T| - \alpha I)$,  $N_j := N(|T| - \beta_j I)$, $j =1,2,\dots,m$ for $m \in \mathbb{N}$.
	
	Then by Lemma \ref{quasinormalreducingsubspace}, we have each $M_i, N_j$ and $H_\infty$ reduces $W$ as well as $T$.
	Let $U_i = W|_{M_i} $, $V_j = W|_{N_j}$ and $V_\infty = W|_{H_\infty}$.
	
	If $\alpha_i \neq 0$ or $\beta_j \neq 0$, then the corresponding $M_i \subset N(W)^\perp$ and $N_j \subset N(W)^\perp$. Hence $U_i$ and $V_j$ are isometries on $M_i$ and $N_j$, respectively. Since $\alpha_i, \beta_j \in \pi_{00}(|T|)$, we get $U_i \in \mathcal{B}(M_i)$ and $V_j \in \mathcal{B}(N_j)$ are unitary operators.	
	
	If $\beta_{j_0} = 0$ for some $j_0 \in \{1,2,\dots,m\}$, then $N_{j_0} = N(T)$ and hence $V_{j_0} =I|_{N(T)}$.
	
	Let $H_1 := \displaystyle \bigoplus_{i = 1}^n M_i$, $n \in \mathbb{N} \cup \{\infty\}$ and $H_2 := \displaystyle \bigoplus_{j = 1}^m N_j$, $m \in \mathbb{N}$. Clearly $H_2$ is finite dimensional. Then $T|_{H_1} = \displaystyle \bigoplus_{i = 1}^n \alpha_i U_i$ and $T|_{H_2} = \displaystyle \bigoplus_{j = 1}^m \beta_j V_j$. Since $|T| = \displaystyle \bigoplus_{\alpha \in \sigma(|T|)} \alpha I_{\alpha}$, where $I_{\alpha}$ is the identity operator on $N(|T|-\alpha I)$, we get $H = H_1 \oplus H_\infty \oplus H_2$.
	
	We have the following cases.
	
	Case$(1)$ $\alpha$ is an eigenvalue of $\sigma(|T|)$ with infinite multiplicity but not the limit point of $\sigma(|T|)$.\\
	Since $\alpha \neq 0$, we get $H_\infty \subset N(W)^\bot$. Hence $V_\infty \in \mathcal{B}(H_\infty)$ is an isometry. So we get $T =  \displaystyle \bigoplus_{i = 1}^{n} \alpha_i U_i \bigoplus \alpha V_\infty \bigoplus_{j = 1}^{m} \beta_jV_j $, $n \in \mathbb{N}$, $m \in \mathbb{N}$,
	
	Case$(2)$ $\alpha$ is the limit point of $\sigma(|T|)$ and $\alpha \notin \sigma_{p}(|T|)$.\\ Then $H_\infty =0$, so $V_\infty =0$ and in this case, we get $T =  \displaystyle \bigoplus_{i = 1}^{\infty} \alpha_i U_i  \bigoplus_{j = 1}^{m} \beta_jV_j $, $m \in \mathbb{N}$.
	
	Case$(3)$ $\alpha$ is both the limit point of $\sigma(|T|)$ as well as an eigenvalue with infinite multiplicity.\\ Then we get $T = \displaystyle \bigoplus_{i = 1}^{\infty} \alpha_i U_i \bigoplus \alpha V_\infty \bigoplus_{j = 1}^{m} \beta_jV_j$,  $m \in \mathbb{N}$.
	
	Case$(4)$ $\alpha$ is the limit point of $\sigma(|T|)$ and an eigenvalue with finite multiplicity.\\ Then a representation of $T$ can be obtained similar to Case $(3)$. As $H_\infty$ is finite dimensional, we get $V_\infty \in \mathcal{B}(H_\infty)$ is a unitary operator.
	
	Moreover, if $m(T) = \alpha = \|T\|$, then  $T = \alpha V_\infty$.
\end{proof}

\begin{corollary}
	Let $T \in \mathcal{AN}(H)$ be quasinormal and $\alpha \in \sigma_{ess}(|T|)$. If any one of the following conditions hold, then $T$ is normal.
	\begin{enumerate}
		\item $\alpha \notin \sigma_{p}(|T|)$.
		\item $\alpha \in \sigma_{p}(|T|)$ with $N(|T| -\alpha I)$ is finite dimensional.
	\end{enumerate}
\end{corollary}
\begin{proof}
	Proof follows from Case$(2)$ and Case$(4)$ of Theorem \ref{ANquasinormal}.
\end{proof}

In \cite{BALA1}, a representation of normal $\mathcal{AM}$ operators was studied. Now, we give a representation of quasinormal $\mathcal{AM}$-operators. The following theorem generalizes the result of \cite{BALA1}.

\begin{theorem}\label{AMquasinormal}
	Let $T \in \mathcal{AM}(H)$ be quasinormal. Then there exists reducing subspaces $H_\infty, H_1$ and $H_2$ with $\text{dim}(H_1) < \infty$ such that $H = H_1 \oplus  H_\infty  \oplus H_2$ and \begin{equation}
	T = \displaystyle \bigoplus_{i = 1}^{n} \alpha_i U_i \bigoplus \alpha V_\infty \bigoplus_{j = 1}^{m} \beta_jV_j
	\end{equation} for some $n \in \mathbb{N}$, $m \in \mathbb{N}  \cup \{\infty\}$, where
	\begin{enumerate}
		\item [(i)]  $\sigma_{ess}(|T|) = \{\alpha\}, \alpha \geq 0$, $H_\infty = N(|T| - \alpha I)$  and $V_\infty \in \mathcal{B}(H_\infty)$ is an isometry.
		
		\item [(ii)] $\{\alpha_i\}_{i =1}^{n} = (\alpha, \|T\|] \cap \sigma(|T|)$, $H_1 = \displaystyle \bigoplus_{i = 1}^{n} N(|T| - \alpha_i I)$ and $U_i \in \mathcal{B}(N(|T| - \alpha_iI))$ is a unitary operator for all $i =1,2, \dots, n$, $n \in \mathbb{N}$.
		\item [(iii)] $ \{\beta_j\}_{j =1}^{m} = [m(T), \alpha) \cap \sigma(|T|)$, $H_2 = \displaystyle \bigoplus_{j = 1}^{m} N(|T| - \beta_j I)$ and $V_j \in \mathcal{B}(N(|T| - \beta_jI))$ is a unitary operator for all $j =1,2, \dots, m$, $m \in \mathbb{N} \cup \{\infty\}$.
		
	\end{enumerate}
\end{theorem}
\begin{proof}
	As $T \in \mathcal{AM}(H)$, we have $|T| \in \mathcal{AM}(H)$ \cite[Theorem 5.14]{GAN}. By \cite[Theorem 3.10]{BALA1}, we get $\sigma_{ess}(|T|)$ is a singleton set, say $\{\alpha\}$ and $(\alpha, \|T\|]$ contains atmost finitely many spectral points of $|T|$.
	
	Let  $[m(T), \alpha) \cap \sigma(|T|) = \{\beta_j\}_{j=1}^{m}, m \in \mathbb{N} \cup \{\infty\}$ and $(\alpha, \|T\|] \cap \sigma(|T|) = \{\alpha_i\}_{i=1}^{n}, n \in \mathbb{N}$.

	\hspace{1cm}
	\begin{tikzpicture}
	\draw[orange, very thick](5,-.3)--(5,.3);
	\filldraw(5,-.3)  node[anchor=west] {$\alpha_1$};
	\draw[very thick]  (-5,0)--(5,0);
	\filldraw [gray] (0,0) circle (2.5pt);
	
	\filldraw(0,-.3)  node[anchor=north] {$\alpha$};
	\draw [orange, very thick](-5,-.3)--(-5,.3);
	\filldraw(-5,-.3) node[anchor=east] {$\beta_1$};
	\draw (4,-.2)--(4,.2) node[anchor=south] {$\alpha_2$};
	\filldraw(2.7,-.3)  node[anchor=north] {Finitely many};
	\draw (2.2,-.2)--(2.2,.2) node[anchor=south] {$.$};
	\draw (3,-.2)--(3,.2) node[anchor=south] {$.$};
	\draw (3.5,-.2)--(3.5,.2) node[anchor=south] {$\alpha_3$};
	\draw (2,-.2)--(2,.2) node[anchor=south] {$.$};
	\draw (1,-.2)--(1,.2) node[anchor=south] {$\alpha_{n-1}$};
	\draw (.3,-.2)--(.3,.2) node[anchor=south] {$\alpha_n$};
	
	\draw (-.11,-.2)--(-.11,.2);
	\draw (-.15,-.2)--(-.15,.2);
	\draw (-.21,-.2)--(-.21,.2);
	\draw (-.28,-.2)--(-.28,.2);
	\draw (-.37,-.2)--(-.37,.2);
	\draw (-.47,-.2)--(-.47,.2);	
	\draw (-.58,-.2)--(-.58,.2);
	\draw (-.7,-.2)--(-.7,.2);
	\draw (-.83,-.2)--(-.83,.2);
	\draw (-.97,-.2)--(-.97,.2);
	\draw (-1.12,-.2)--(-1.12,.2)node[anchor=south] {$.$};	
	\draw (-1.28,-.2)--(-1.28,.2)node[anchor=south] {$.$};
	\draw (-1.45,-.2)--(-1.45,.2)node[anchor=south] {$.$};
	\draw (-1.63,-.2)--(-1.63,.2)node[anchor=south] {$.$};
	\draw (-1.82,-.2)--(-1.82,.2) node[anchor=south] {$.$};
	\draw (-2.02,-.2)--(-2.02,.2) node[anchor=south] {$.$};
	\draw (-2.5,-.2)--(-2.5,.2) node[anchor=south] {$\beta_4$};
	\draw (-3.12,-.2)--(-3.12,.2) node[anchor=south] {$\beta_3$};
	\draw (-3.9,-.2)--(-3.9,.2)node[anchor=south] {$\beta_2$};	
	\end{tikzpicture}
	
	\hspace{4cm} Spectral Diagram of $|T| \in \mathcal{AM}(H)$
	
	\vspace{.5cm}By following the similar steps as in Theorem \ref{ANquasinormal}, we get reducing subspaces $H_1, H_\infty, H_2$ such that $H = H_1 \oplus H_\infty \oplus H_2 $, where   $H_1 = \displaystyle \bigoplus_{i = 1}^{n} N(|T| - \alpha_i I), H_\infty = N(|T|-\alpha I)$ and $H_2 = \displaystyle \bigoplus_{j = 1}^{m} N(|T| - \beta_j I), n \in \mathbb{N}, m \in \mathbb{N} \cup \{\infty\}$ and with respect to this decomposition $T$ can be written as \[T = \displaystyle \bigoplus_{i = 1}^{n} \alpha_i U_i \bigoplus \alpha V_\infty \bigoplus_{j = 1}^{m} \beta_jV_j,\] where $U_i$ is a unitary operator on $N(|T|- \alpha_i I)$, $V_j$ is a unitary operator on $N(|T|-\beta_j I)$ and $V_\infty \in \mathcal{B}(H_\infty)$ is an isometry. As $\alpha_i \in \pi_{00}(|T|)$ for all $i = 1,2,\dots,n$, we get $H_1$ is finite dimensional.
\end{proof}

\begin{corollary}
	Let $T \in \mathcal{AM}(H)$ be quasinormal and $\alpha \in \sigma_{ess}(|T|)$. Then any of the following conditions imply that $T$ is normal.
	\begin{enumerate}
		\item $\alpha \notin \sigma_{p}(|T|)$.
		\item $\alpha \in \sigma_{p}(|T|)$ with finite dimensional eigenspace.
	\end{enumerate}
\end{corollary}

Hence we now give a representation of quasinormal operators in $\overline{\mathcal{AN}(H)}$ which contains both $\mathcal{AN}$ and $\mathcal{AM}$-operators.
\begin{theorem}\label{ANquasinormal}
	Let $T \in \overline{\mathcal{AN}(H)}$ be quasinormal. Then there exists reducing subspaces $H_\infty, H_1$ and $H_2$ such that $H = H_1 \oplus  H_\infty  \oplus H_2$ and \begin{equation}
	T = \displaystyle \bigoplus_{i = 1}^{n} \alpha_i U_i \bigoplus \alpha V \bigoplus_{j = 1}^{m} \beta_jV_j,
	\end{equation} for some $n, m \in \mathbb{N} \cup \{\infty\}$, where
	\begin{enumerate}
		\item [(i)]  $\sigma_{ess}(|T|) = \{\alpha\}, \alpha \geq 0$, $H_\infty = N(|T| - \alpha I)$  and $V_\infty \in \mathcal{B}(H_\infty)$ is an isometry.
		
		\item [(ii)] $(\alpha, \|T\|] \cap \sigma(|T|)= \{\alpha_i\}_{i =1}^{n}$, $H_1 = \displaystyle \bigoplus_{i = 1}^{n} N(|T| - \alpha_i I)$ and $U_i \in \mathcal{B}(N(|T| - \alpha_iI))$ is a unitary operator for all $i =1,2, \dots, n$, $n \in \mathbb{N} \cup \{\infty\}$.
		\item [(iii)] $[m(T), \alpha) \cap \sigma(|T|) =  \{\beta_j\}_{j =1}^{m}$, $H_2 = \displaystyle \bigoplus_{j = 1}^{m} N(|T| - \beta_j I)$ and $V_j \in \mathcal{B}(N(|T| - \beta_jI))$ is a unitary operator for all $j =1,2, \dots, m$, $m \in \mathbb{N} \cup \{\infty\}$.
	\end{enumerate}
\end{theorem}

\begin{proof}
	Let $T= W|T|$ be the polar decomposition of $T$.
	
	As $|T| \in \overline{\mathcal{AN}(H)}$, we have $|T| = \alpha I - K_1 +K_2$, where $\alpha \geq 0, K_1, K_2 \in \mathcal{K}(H)_+$ with $K_1K_2 =0$ and $K_1 \leq \alpha I$ by \cite[Theorem 4.2]{RAMSSS}. So by Proposition \ref{prop1} (\ref{Tdecomposition}), we have
	\item \label{Tdecomposition}
	\[|T| =
	\begin{blockarray}{ccc}
	N(K_1) & N(K_1)^\bot\\
	\begin{block}{(cc)c}
	\alpha I_{N(K_1)} +	\widetilde{K_2} & 0 & N(K_1) \\
	{}\\
	0 & \alpha I_{N(K_1)^{\bot}} -	\widetilde{K_1} & N(K_1)^\bot\\
	\end{block}
	\end{blockarray},\]
	
	where $\widetilde{K_2} = K_2|_{N(K_1)}$ and $\widetilde{K_1} = K_1|_{N(K_1)^{\bot}}$.
	
	Since $T$ is quasinormal, by \cite[Proposition 6.4, Page 444]{KUB} we have $W|T| = |T|W$ which further implies that $W(K_2 -K_1) = (K_2 -K_1)W$. So $W(K_2-K_1)^2 = (K_2-K_1)^2W$. As $K_1K_2 =0$, we get $W(K^2_2+ K^2_1) = (K^2_2+ K^2_1)W$. This implies $W(K_2 +K_1)^2 = (K_2 +K_1)^2W$. Hence $W$ commutes with $K_1 + K_2$ which is a positive square root of $(K_2 +K_1)^2$.
	
	Therefore we get $WK_1 = K_1W$ and $WK_2 = K_2W$. This implies $N(K_1)$ is invariant under $W$. Further, we also have $W^*|T| = |T|W^*$. Hence $N(K_1)$ reduces $W$. Let
	
	\[W =
	\begin{blockarray}{ccc}
	N(K_1) & N(K_1)^\bot\\
	\begin{block}{(cc)c}
	W_1 & 0 & N(K_1) \\
	0 & W_2 & N(K_1)^\bot\\
	\end{block}
	\end{blockarray}.\]
	
	Then \[T =
	\begin{blockarray}{ccc}
	N(K_1) & N(K_1)^\bot\\
	\begin{block}{(cc)c}
	T_1 & 0 & N(K_1) \\
	0 & T_2 & N(K_1)^\bot\\
	\end{block}
	\end{blockarray},\]
	where $T_1 = W_1(\alpha I_{N(K_1)} +	\widetilde{K_2}) \in \mathcal{AN}(H)$ , $T_2 = W_2(\alpha I_{N(K_1)^{\bot}} -	\widetilde{K_1}) \in \mathcal{AM}(H)$ and both $T_1$ and $T_2$ are quasinormal.
	
	Let $\sigma(|T_1|) = \{\alpha_i\}_{i=1}^{n}$ and $\sigma(|T_2|) =  \{\beta_j\}_{j=1}^{m}$, $n, m \in \mathbb{N} \cup \{\infty\}$, where $\{\alpha_i\}$ decreases to $\alpha$ and $\{\beta_j\}$ increases to $\alpha$, if they are infinite. Since by \cite[Theorem 4.5]{RAMSSS}, we have $\sigma_{ess}(|T|)$ is a singleton set, say $\{\alpha\}$, we get $\sigma_{ess}(|T_1|) = \sigma_{ess}(|T_2|) = \{\alpha\}$. If $\alpha \in \sigma_{p}(|T|)$, then we let $\alpha \in \sigma_{p}(|T_2|)$ and $\alpha \notin \sigma_{p}(|T_1|)$
	
	\hspace{1cm}
	\begin{tikzpicture}
	\draw[orange, very thick](5,-.3)--(5,.3);
	\filldraw(5,-.3)  node[anchor=west] {$\alpha_1$};
	\draw[very thick]  (-5,0)--(5,0);
	\filldraw [gray] (0,0) circle (2.5pt);
	\filldraw(0,-.3)  node[anchor=north] {$\alpha$};
	\draw [orange, very thick](-5,-.3)--(-5,.3);
	\filldraw(-5,-.3) node[anchor=east] {$\beta_1$};
	\draw (-.11,-.2)--(-.11,.2);
	\draw (-.15,-.2)--(-.15,.2);
	\draw (-.21,-.2)--(-.21,.2);
	\draw (-.28,-.2)--(-.28,.2);
	\draw (-.37,-.2)--(-.37,.2);
	\draw (-.47,-.2)--(-.47,.2);	
	\draw (-.58,-.2)--(-.58,.2);
	\draw (-.7,-.2)--(-.7,.2);
	\draw (-.83,-.2)--(-.83,.2);
	\draw (-.97,-.2)--(-.97,.2);
	\draw (-1.12,-.2)--(-1.12,.2) node[anchor=south] {.};	
	\draw (-1.28,-.2)--(-1.28,.2) node[anchor=south] {.};
	\draw (-1.45,-.2)--(-1.45,.2) node[anchor=south] {.};
	\draw (-1.63,-.2)--(-1.63,.2) node[anchor=south] {.};
	\draw (-1.82,-.2)--(-1.82,.2) node[anchor=south] {.};
	\draw (-2.02,-.2)--(-2.02,.2) node[anchor=south] {$\beta_5$};
	\draw (-2.5,-.2)--(-2.5,.2) node[anchor=south] {$\beta_4$};
	\draw (-3.12,-.2)--(-3.12,.2) node[anchor=south] {$\beta_3$};
	\draw (-3.7,-.2)--(-3.7,.2) node[anchor=south] {$\beta_2$};
	
	\draw (.10,-.2)--(.10,.2);
	\draw (.15,-.2)--(.15,.2);
	\draw (.21,-.2)--(.21,.2);
	\draw (.28,-.2)--(.28,.2);
	\draw (.37,-.2)--(.37,.2);
	\draw (.47,-.2)--(.47,.2);	
	\draw (.58,-.2)--(.58,.2);
	\draw (.7,-.2)--(.7,.2);
	\draw (.83,-.2)--(.83,.2) node[anchor=south] {$.$};
	\draw (.97,-.2)--(.97,.2) node[anchor=south] {$.$};
	\draw (1.12,-.2)--(1.12,.2)node[anchor=south] {$.$};	
	\draw (1.28,-.2)--(1.28,.2) node[anchor=south] {$.$};
	\draw (1.45,-.2)--(1.45,.2) node[anchor=south] {$.$};
	\draw (1.63,-.2)--(1.63,.2) node[anchor=south] {$.$};
	\draw (1.82,-.2)--(1.82,.2) node[anchor=south] {$.$};
	\draw (2.02,-.2)--(2.02,.2) node[anchor=south] {$\alpha_5$};
	\draw (2.5,-.2)--(2.5,.2) node[anchor=south] {$\alpha_4$};
	\draw (3.12,-.2)--(3.12,.2) node[anchor=south] {$\alpha_3$};
	\draw (3.9,-.2)--(3.9,.2) node[anchor=south] {$\alpha_2$};
	\end{tikzpicture}
	
	\hspace{4cm} Spectral Diagram of $|T| \in \overline{\mathcal{AN}(H)}$
	
	\vspace{0.5cm} As $m(T_1) = \alpha$ and $\alpha \notin \sigma_{p}(|T_1|)$, by Theorem \ref{ANquasinormal},  we get $N(K_1)= \displaystyle \bigoplus_{i=1}^{n} N(|T|-\alpha_i I) $ and $T_1 = \displaystyle \bigoplus_{i=1}^{n} \alpha_iU_i$ where $U_i = W|_{N(|T|-\alpha_i I)}$ is a unitary operator for all $i=1,2,\dots, n$, $n \in \mathbb{N} \cup \{\infty\}$.  Similarly, as $\alpha = \|T_2\|$, by Theorem \ref{AMquasinormal}, we get $N(K_1)^\bot = \displaystyle \bigoplus_{i=1}^{m} N(|T|-\beta_j I) \bigoplus N(|T|-\alpha I)$ and $T_2 = \displaystyle \bigoplus_{j=1}^{m} \beta_jV_j \bigoplus \alpha V$, where $V_j = W|_{N(|T|-\beta_j I)}$ is a unitary operator for all $i=1,2,\dots, m$, $m \in \mathbb{N} \cup \{\infty\}$ and $V_\infty = W|_{N(|T|-\alpha I)}$ is an isometry.  Taking $H_1 := \displaystyle \bigoplus_{i=1}^{n} N(|T|-\alpha_i I)$, $H_\infty := N(|T|-\alpha I)$ and $H_2 := \displaystyle \bigoplus_{i=1}^{m} N(|T|-\beta_j I)$, we get $H = H_1 \oplus H_\infty \oplus H_2$ and $T = \displaystyle \bigoplus_{i=1}^{n} \alpha_iU_i \bigoplus \alpha V_\infty \bigoplus_{j=1}^{m} \beta_jV_j$ , $n, m \in \mathbb{N} \cup \{\infty\}$.
\end{proof}

\begin{corollary}
	Let $T \in \overline{\mathcal{AN}(H)}$ be quasinormal with $\alpha \in \sigma_{ess}(|T|)$. Then any of the following conditions imply that $T$ is normal.
\begin{enumerate}
	\item $\alpha \notin \sigma_{p}(|T|)$.
	\item $\alpha \in \sigma_{p}(|T|)$ with finite dimensional eigenspace.
\end{enumerate}
\end{corollary}

\section{Hyponormal operators in $\overline{\mathcal{AN}(H)}$}

In \cite{BALA2}, the authors have given a representation for hyponormal $\mathcal{AN}$-operators. In this section, we first give a representation of hyponormal $\mathcal{AM}$-operators followed by a representation of hyponormal operators in $\overline{\mathcal{AN}(H)}$. The results in this section generalizes the results of \cite{BALA2}.

\begin{lemma} \label{hyponormal1}
	Let $T \in \mathcal{AM}(H)$ be hyponormal with $\sigma_{ess}(|T|) = \{\|T\|\}$. If $\|T\|$ is an eigenvalue of $|T|$, then  there exists subspaces $H_1$ and $H_2$ such  that $H = H_1 \oplus H_2$ and  with respect to $H_1 \oplus H_2$,  $T$ can be represented as follows.
		\[T =
		\begin{blockarray}{ccc}
		H_1 & H_2 \\
		\begin{block}{(cc)c}
	\|T\| V_1 & A & H_ 1 \\
	     0 & B & H_2 \\
		\end{block}
		\end{blockarray},
		\]
where $H_1 = N(|T|- \|T\| I), H_2 = \displaystyle \bigoplus_{j = 1}^{m} N(|T|-\beta_j I)$ for some $m \in \mathbb{N} \cup \{\infty\}$ if $\pi_ {00}(|T|) = \{\beta_j\}_{j=1}^{m}$ and
\begin{enumerate}
	\item [i)] $V_1$ is an isometry on $H_1$.
	\item [ii)]$A \in \mathcal{B}(H_2, H_1)$ and  $B \in \mathcal{B}(H_2)$ satisfying $V^*_1A = 0$, $A^*A +B^*B = \displaystyle \bigoplus_{j=1}^{m}\beta^2_j I_{N(|T|- \beta_j I)}$ and $B B^* \leq \displaystyle \bigoplus_{j=1}^{m}\beta^2_j I_{N(|T|- \beta_j I)}$.
\end{enumerate}
If $\pi_{00}(|T|)$ is empty, then $T = \|T\|V_1$.
\end{lemma}
\begin{proof}
	The proof is inspired by \cite[Proposition 3.3]{BALA2}.
		Let $T = V|T|$ be a factorization of $T$ such that $V$ is an isometry \cite[Page 4]{XIA}. As $\|T\|$ is an eigenvalue of $|T|$, $H_1 \neq \{0\}$. Since $T$ is hyponormal, by \cite[Theorem 3]{LEE},  we have that $H_1$ is invariant under $T$ and $\frac{T}{\|T\|}$ is an isometry on $H_1$. Also, if $x \in H_1$, then $|T|x = \|T\|x$. So
	\begin{equation*}
	Vx = V \frac{|T|x}{\|T\|} = \frac{Tx}{\|T\|} \in H_1.
	\end{equation*}
	This implies $H_1$ is invariant under $V$ as well.
Hence with respect to $H_1 \oplus H_2$, we can write $V$ and $|T|$ as
\[V =
\begin{blockarray}{ccc}
H_1 & H_2  \\
\begin{block}{(cc)c}
V_1 & V_2 & H_1 \\
0 & V_3 & H_2 \\
\end{block}
\end{blockarray},
 \
|T| =
\begin{blockarray}{ccc}
H_1 & H_2  \\
\begin{block}{(cc)c}
\|T\| I_{H_1} & 0 & H_1 \\
0 & T_1& H_2 \\
\end{block}
\end{blockarray},
\]
where $V_1$ is an isometry on $H_1, V_2 = P_{H_1}V|_{H_2}, V_3 = P_{H_2}V|_{H_2}$ and $T_1 = \displaystyle \bigoplus_{j=1}^{m}\beta_jI_{N(|T|-\beta_jI)}$, where $\{\beta_j\}_{j=1}^{m} \in \pi_{00}(|T|)$ for some $m \in \mathbb{N} \cup \{\infty\}$. If $\pi_{00}(|T|)$ is empty, then we get $T = \|T\|V_1$.

If $\pi_{00}(|T|)$ is non-empty, then $H_2 \neq \{0\}$ and we have
\begin{equation*}
\begin{split}
T &=
\begin{pmatrix}
V_1  &  V_2   \\
0  &  V_3
\end{pmatrix}
\begin{pmatrix}
\|T\| I_{H_1} &  0  \\
0  &  T_1
\end{pmatrix} \\
& = \begin{pmatrix}
\|T\| V_1 &  V_2 T_1 \\
0  & V_3 T_1
\end{pmatrix}
=
\begin{pmatrix}
\|T\| V_1 &  A \\
0  & B
\end{pmatrix} ,
\end{split}
\end{equation*}

where $A = V_2T_1$ and $B = V_3T_1$.

Substituting in  $T^*T = |T|^2$, 
\begin{equation*}
\begin{split}
 \begin{pmatrix}
\|T\|^2 V^*_1V_1 & \|T\|V^*_1A \\
\|T\|A^*V_1 & A^*A +B^*B
\end{pmatrix}
= 
\begin{pmatrix}
\|T\|^2 I_{H_1} &  0  \\
0  &  T^2_1
\end{pmatrix} \\
\end{split}.
\end{equation*}

Equating on both sides, we get $V^*_1A =0$ and $A^*A +B^*B = T^2_1 = \displaystyle \bigoplus_{j=1}^{m}\beta^2_jI_{N(|T|-\beta_jI)}.$ 

Now, by the definition of hyponormality of $T$, we have
\begin{equation*}
\begin{split}
0 & \leq  T^*T -TT^*\\
&=  \begin{pmatrix}
\|T\|^2 V^*_1V_1 & \|T\|V^*_1A \\
\|T\|A^*V_1 & A^*A +B^*B
\end{pmatrix}
-
\begin{pmatrix}
\|T\|^2 V_1V^*_1 + AA^* & AB^* \\
BA^*  & BB^*
\end{pmatrix} \\
&=  \begin{pmatrix}
\|T\|^2 I_{H_1}- \|T\|^2 P_{R(V_1)}-AA^* & \|T\|V^*_1A-AB^* \\
\|T\|A^*V_1-BA^*  & A^*A +B^*B-BB^*
\end{pmatrix}\\
&= \begin{pmatrix}
\|T\|^2 P_{N(V^*_1)}-AA^* & -AB^* \\
-BA^*  & A^*A +B^*B-BB^*
\end{pmatrix}.
\end{split}
\end{equation*}

Hence $A^*A +B^*B-BB^* \geq 0$ implies $BB^* \leq \bigoplus_{j=1}^{m}\beta^2_jI_{N(|T|-\beta_jI)}.$

 \end{proof}

Using the above lemma, we give a representation of hyponormal $\mathcal{AM}$-operators.

\begin{theorem}\label{AMHypo}
	Let $T \in \mathcal{AM}(H)$ be hyponormal. Then there exists subspaces $H_0, H_1, H_2$  with dim$(H_0) < \infty$ such that $H = H_0 \oplus H_1 \oplus H_2$ and $T$ has the representation as given below.
		\begin{equation} \label{AMhypomatrix}
	T =
	\begin{blockarray}{cccc}
	H_0 & H_1 & H_2 \\
	\begin{block}{(ccc)c}
	V_0 & 0 & 0 & H_0 \\
	0 & \beta V_1 & A & H_1 \\
	0 & 0 & B & H_2\\
	\end{block}
	\end{blockarray},
	\end{equation}
	where $\sigma_{ess}(|T|) = \{\beta\}$,  $\beta \geq 0$ and
	\begin{enumerate}
		\item[(i)] $V_0 = \displaystyle \bigoplus_{i=1}^{n} \alpha_i U_i$ is a finite rank normal operator such that $U_i$ is unitary on $N(|T| - \alpha_i I)$ if $(\beta, \|T\|] \cap \sigma(|T|) = \{\alpha_i\}_{i=1}^n$, $n \in \mathbb{N}$.
		
		If $\beta = \|T\|$, then $H_0 = \{0\}$.
		
		\item[(ii)]$V_1$ is an isometry on $H_1$ if $H_1 \neq \{0\}$, otherwise $V_1 =0$.
		\item [(iii)] $A \in \mathcal{B}(H_2, H_1)$ and $B \in \mathcal{B}(H_2)$ satisfy $V^*_1A = 0, A^*A +B^*B=  \displaystyle \bigoplus_{j=1}^{m} \beta^2_jI_{N(|T|-\beta_jI)}$ and $BB^* \leq  \displaystyle \bigoplus_{j=1}^{m}\beta^2_jI_{N(|T|-\beta_jI)}$, where $[m(T), \beta) \cap \sigma(|T|) = \{\beta_j\}_{j=1}^m$ for some $m \in \mathbb{N} \cup \{\infty\}$.
		
		If $\beta = m(T)$, then $H_2 = \{0\}$ and $ T = \|T\|V_1$ if $\|T\|= \beta = m(T)$.
	\end{enumerate}
\begin{proof}
	As $T \in \mathcal{AM}(H)$, we have $|T| \in \mathcal{AM}(H)$ by \cite[Theorem 5.14]{GAN}. Hence by \cite[Theorem 3.10]{BALA1}, $\sigma_{ess}(|T|)$ is a singleton set, say $\{\beta\}$ and $(\beta, \|T\|]$ has atmost finitely many spectral points of $|T|$.

	Let $(\beta, \|T\|] \cap \sigma(|T|) = \{\alpha_i\}_{i=1}^{n}, n \in \mathbb{N}$ and  $[m(T), \beta) \cap \sigma(|T|) = \{\beta_j\}_{j=1}^{m}, m \in \mathbb{N} \cup \{\infty\}$.

\hspace{1cm}
\begin{tikzpicture}
\draw[orange, very thick](5,-.3)--(5,.3);
\filldraw(5,-.3)  node[anchor=west] {$\alpha_1$};
\draw[very thick]  (-5,0)--(5,0);
\filldraw [gray] (0,0) circle (2.5pt);

\filldraw(0,-.3)  node[anchor=north] {$\beta$};
\draw [orange, very thick](-5,-.3)--(-5,.3);
\filldraw(-5,-.3) node[anchor=east] {$\beta_1$};
\draw (4,-.2)--(4,.2) node[anchor=south] {$\alpha_2$};
\filldraw(2.7,-.3)  node[anchor=north] {Finitely many};
\draw (2.2,-.2)--(2.2,.2) node[anchor=south] {$.$};
\draw (3,-.2)--(3,.2) node[anchor=south] {$.$};
\draw (3.5,-.2)--(3.5,.2) node[anchor=south] {$\alpha_3$};
\draw (2,-.2)--(2,.2) node[anchor=south] {$.$};
\draw (1,-.2)--(1,.2) node[anchor=south] {$\alpha_{n-1}$};
\draw (.3,-.2)--(.3,.2) node[anchor=south] {$\alpha_n$};

\draw (-.11,-.2)--(-.11,.2);
\draw (-.15,-.2)--(-.15,.2);
\draw (-.21,-.2)--(-.21,.2);
\draw (-.28,-.2)--(-.28,.2);
\draw (-.37,-.2)--(-.37,.2);
\draw (-.47,-.2)--(-.47,.2);	
\draw (-.58,-.2)--(-.58,.2);
\draw (-.7,-.2)--(-.7,.2);
\draw (-.83,-.2)--(-.83,.2);
\draw (-.97,-.2)--(-.97,.2);
\draw (-1.12,-.2)--(-1.12,.2)node[anchor=south] {$.$};	
\draw (-1.28,-.2)--(-1.28,.2)node[anchor=south] {$.$};
\draw (-1.45,-.2)--(-1.45,.2)node[anchor=south] {$.$};
\draw (-1.63,-.2)--(-1.63,.2)node[anchor=south] {$.$};
\draw (-1.82,-.2)--(-1.82,.2) node[anchor=south] {$.$};
\draw (-2.02,-.2)--(-2.02,.2) node[anchor=south] {$.$};
\draw (-2.5,-.2)--(-2.5,.2) node[anchor=south] {$\beta_4$};
\draw (-3.12,-.2)--(-3.12,.2) node[anchor=south] {$\beta_3$};
\draw (-3.9,-.2)--(-3.9,.2)node[anchor=south] {$\beta_2$};	
\end{tikzpicture}

\hspace{4cm} Spectral Diagram of $|T| \in \mathcal{AM}(H)$

\vspace{0.5cm} Without loss of generality let $\alpha_1 = \|T\|$.
Since $T$ is hyponormal, by \cite[Theorem 3]{LEE},  we have $N(|T|-\alpha_1 I)$ is invariant under $T$ as well as $V$. Moreover $\frac{T}{\alpha_1}$ is an isometry on $N(|T|-\alpha_1 I)$. As $\alpha _1 \in \pi_{00}(|T|)$, we get $N(|T|-\alpha_1 I)$ is finite dimensional. Hence $T = \alpha_1U_1$ on $N(|T|-\alpha_1 I)$, where $U_1 \in \mathcal{B}(N(|T|-\alpha_1 I))$ is a unitary operator. Therefore by \cite[Lemma 5]{STA}, $N(|T|-\alpha_1 I)$ reduces $T$. Now, $T|_{(N(|T|-\alpha_1 I))^\bot} \in \mathcal{AM}((N(|T|-\alpha_1 I))^\bot)$ with  $\|T|_{(N(|T|-\alpha_1 I))^\bot}\| = \alpha_2$. Continuing the same process as above, we get $T = \displaystyle \bigoplus_{i=1}^{n} \alpha_i U_i$ on $H_0 = \displaystyle \bigoplus_{i=1}^{n} N(|T|- \alpha_i I)$, $n \in \mathbb{N}$. Clearly $H_0$ is finite dimensional, hence $V_0 = T|_{H_0} = \displaystyle \bigoplus_{i=1}^{n} \alpha_i U_i$ is a finite rank normal operator. If $\beta = \|T\|$, then $(\beta, \|T\|] \cap \sigma(|T|)$ is empty, hence $H_0 = \{0\}$

Now, $T|_{{H_0}^\bot} \in \mathcal{AM}({H_0}^\bot)$ is hyponormal, with $\|T|_{{H_0}^\bot}\| = \beta$ and $\sigma_{ess}(T|_{{H_0}^\bot}) = \{\beta\}$. Therefore by Lemma \ref{hyponormal1}, we get $H^{\bot}_0 = H_1 \oplus H_2$, where $H_1 = N(|T|- \beta I)$ and  $H_2 = \displaystyle \bigoplus_{j = 1}^{m} N(|T|-\beta_j I)$, $m \in \mathbb{N} \cup \{\infty\}$ and with respect to this decomposition, $T|_{H^{\perp}_0}$ can be written as,

\[T|_{H^{\perp}_0} =
\begin{blockarray}{ccc}
H_1 & H_2 \\
\begin{block}{(cc)c}
\beta V_1 & A & H_ 1 \\
0 & B & H_2 \\
\end{block}
\end{blockarray},
\]
 where $V_1$ is an isometry on $H_1$ and the operators $A, B$ satisfies the required conditions from Lemma \ref{hyponormal1}. Clearly if $\beta \notin \sigma_{p}(|T|)$, then $H_1 = \{0\}$. If $[m(T), \beta) \cap \sigma(|T|)$ is empty, then $H_3 = \{0\}$. Moreover if $m(T) = \beta = \|T\|$, then $T = \|T\|V_1$. Hence we get the representation of $T$ as follows.

 	\begin{equation}
 T =
 \begin{blockarray}{cccc}
 H_0 & H_1 & H_2 \\
 \begin{block}{(ccc)c}
 V_0 & 0 & 0 & H_0 \\
 0 & \beta V_1 & A & H_1 \\
 0 & 0 & B & H_2\\
 \end{block}
 \end{blockarray}.
 \end{equation}

\end{proof}
	
\end{theorem}

Now, we look at a representation of hyponormal operators in $\overline{\mathcal{AN}(H)}$ which contains both hyponormal $\mathcal{AN}$-operators and hyponormal $\mathcal{AM}$-operators.

\begin{theorem} \label{hypo2}
	Let $T \in \overline{\mathcal{AN}(H)}$ be hyponormal. Then there exists subspaces $H_0, H_1, H_2$ such that $H = H_0 \oplus H_1  \oplus H_2$ and $T$ has the following representation.
	
		\begin{equation} \label{Thypo}
	T =
	\begin{blockarray}{cccc}
	H_0 & H_1 & H_2 \\
	\begin{block}{(ccc)c}
	V_0 & 0 & 0 & H_0 \\
	0 & \alpha V_1 & A & H_1 \\
	0 & 0 & B & H_2\\
	\end{block}
	\end{blockarray},
		\end{equation}

where $\sigma_{ess}(|T|) = \{\alpha\}$,  $\alpha \geq 0$ and
\begin{enumerate}
	\item[(i)] $V_0 = \displaystyle \bigoplus_{i=1}^{n} \alpha_i U_i$, where $U_i$ is unitary on $N(|T| - \alpha_i I)$ if $(\alpha, \|T\|] \cap \sigma(|T|) = \{\alpha_i\}_{i=1}^n$ for some $n \in \mathbb{N} \cup \{\infty\}$.

	If $\alpha = \|T\|$, then $H_0 = \{0\}$.
	
\item[(ii)]$V_1$ is an isometry on $H_1$ if $H_1 \neq \{0\}$, otherwise $V_1 =0$.
\item [(iii)] $A \in \mathcal{B}(H_2, H_1)$ and $B \in \mathcal{B}(H_2)$ satisfy $V^*_1A = 0, A^*A +B^*B= \displaystyle \bigoplus_{j=1}^{m} \beta^2_jI_{N(|T| - \beta_j I)}$ and $BB^*= \displaystyle \bigoplus_{j=1}^{m} \beta^2_jI_{N(|T| - \beta_j I)}$, where $[m(T), \alpha) \cap \sigma(|T|) = \{\beta_j\}_{j=1}^m$ for some $m \in \mathbb{N} \cup \{\infty\}$.

If $\alpha = m(T)$, then $H_2 = \{0\}$ and $ T = \|T\|V_1$ if $\|T\|= \alpha = m(T)$.
\end{enumerate}
\end{theorem}
\begin{proof}
Since $T \in \overline{\mathcal{AN}(H)}$, we have $|T| \in \overline{\mathcal{AN}(H)}$ by \cite[Lemma 3.13]{RAMSSS}. Hence  by Theorem \ref{positiveessentialspectrum}, we have $\sigma_{ess}(|T|)$ is a single point, say $\{\alpha\}, \alpha \geq 0$. So $m_e(|T|) = \alpha$,  $(\alpha, \|T\|] \cap \sigma(|T|)$ is atmost countable say $\{\alpha_i\}_{i=1}^{n}$ for some $n \in \mathbb{N} \cup \{\infty\}$ and  $[m(T), \alpha) \cap \sigma(|T|)$ is also atmost countable say $\{\beta_j\}_{j=1}^{m}$ for some $m \in \mathbb{N} \cup \{\infty\}$ .

\hspace{1cm}
\begin{tikzpicture}
\draw[orange, very thick](5,-.3)--(5,.3);
\filldraw(5,-.3)  node[anchor=west] {$\alpha_1$};
\draw[very thick]  (-5,0)--(5,0);
\filldraw [gray] (0,0) circle (2.5pt);
\filldraw(0,-.3)  node[anchor=north] {$\alpha$};
\draw [orange, very thick](-5,-.3)--(-5,.3);
\filldraw(-5,-.3) node[anchor=east] {$\beta_1$};
\draw (-.11,-.2)--(-.11,.2);
\draw (-.15,-.2)--(-.15,.2);
\draw (-.21,-.2)--(-.21,.2);
\draw (-.28,-.2)--(-.28,.2);
\draw (-.37,-.2)--(-.37,.2);
\draw (-.47,-.2)--(-.47,.2);	
\draw (-.58,-.2)--(-.58,.2);
\draw (-.7,-.2)--(-.7,.2);
\draw (-.83,-.2)--(-.83,.2);
\draw (-.97,-.2)--(-.97,.2);
\draw (-1.12,-.2)--(-1.12,.2) node[anchor=south] {.};	
\draw (-1.28,-.2)--(-1.28,.2) node[anchor=south] {.};
\draw (-1.45,-.2)--(-1.45,.2) node[anchor=south] {.};
\draw (-1.63,-.2)--(-1.63,.2) node[anchor=south] {.};
\draw (-1.82,-.2)--(-1.82,.2) node[anchor=south] {.};
\draw (-2.02,-.2)--(-2.02,.2) node[anchor=south] {$\beta_5$};
\draw (-2.5,-.2)--(-2.5,.2) node[anchor=south] {$\beta_4$};
\draw (-3.12,-.2)--(-3.12,.2) node[anchor=south] {$\beta_3$};
\draw (-3.7,-.2)--(-3.7,.2) node[anchor=south] {$\beta_2$};

\draw (.10,-.2)--(.10,.2);
\draw (.15,-.2)--(.15,.2);
\draw (.21,-.2)--(.21,.2);
\draw (.28,-.2)--(.28,.2);
\draw (.37,-.2)--(.37,.2);
\draw (.47,-.2)--(.47,.2);	
\draw (.58,-.2)--(.58,.2);
\draw (.7,-.2)--(.7,.2);
\draw (.83,-.2)--(.83,.2) node[anchor=south] {$.$};
\draw (.97,-.2)--(.97,.2) node[anchor=south] {$.$};
\draw (1.12,-.2)--(1.12,.2)node[anchor=south] {$.$};	
\draw (1.28,-.2)--(1.28,.2) node[anchor=south] {$.$};
\draw (1.45,-.2)--(1.45,.2) node[anchor=south] {$.$};
\draw (1.63,-.2)--(1.63,.2) node[anchor=south] {$.$};
\draw (1.82,-.2)--(1.82,.2) node[anchor=south] {$.$};
\draw (2.02,-.2)--(2.02,.2) node[anchor=south] {$\alpha_5$};
\draw (2.5,-.2)--(2.5,.2) node[anchor=south] {$\alpha_4$};
\draw (3.12,-.2)--(3.12,.2) node[anchor=south] {$\alpha_3$};
\draw (3.9,-.2)--(3.9,.2) node[anchor=south] {$\alpha_2$};
\end{tikzpicture}

\hspace{4cm} Spectral Diagram of $|T| \in \overline{\mathcal{AN}(H)}$

 Let $H_0 = \displaystyle \bigoplus_{i=1}^{n} N(|T|- \alpha_i I)$, $H_1 = N(|T|- \alpha I)$ and $H_2 = \displaystyle \bigoplus_{j=1}^{m} N(|T|- \beta_j I)$.

We now consider the following cases which exhaust all the possibilities.

Case$(1)$ $[m(T), \alpha) \cap \sigma(|T|)$ is finite.\\ In this case, $H_2$ is finite dimensional and by \cite[Theorem 2.4]{RAM2}, $T \in \mathcal{AN}(H)$. Hence the required representation of $T$ is obtained from \cite[Theorem 3.5]{BALA2}.

Case$(2)$ $(\alpha, \|T\|] \cap \sigma(|T|)$ is finite.\\ In this case, $H_0$ is finite dimensional and $T \in \mathcal{AM}(H)$ by \cite[Theorem 3.10]{BALA2}. Hence by Theorem \ref{AMHypo}, we get the required structure of $T$.

Case$(3)$ Both $[m(T), \alpha) \cap \sigma(|T|)$ and $(\alpha, \|T\|] \cap \sigma(|T|)$ are countably infinite.

 Following the same steps as in Theorem \ref{AMHypo}, we get $H_0$ reduces $T$, where $\sigma(|T|_{H_0}|) = \{\alpha_i\}_{i=1}^{\infty} \cup \{\alpha\}$ with $\sigma_{ess}(|T|_{H_0}|) = \{\alpha\}$ and $ V_0 = T|_{H_0} = \displaystyle \bigoplus_{i=1}^{\infty} \alpha_i U_i $,
where $U_i\in \mathcal{B}(N(|T|-\alpha_i I))$ is a unitary operator. Hence we get

  \begin{equation*}
  T = \begin{blockarray}{ccc}
  H_0 & H^\bot_0\\
  \begin{block}{(cc)c}
  V_0 & 0& H_0\\
  0 & T_1& H^\bot_0\\
  \end{block}
  \end{blockarray},
  \end{equation*}
  where $T_1 = T|_{H^\bot_0} \in \overline{\mathcal{AN}(H^\perp_0)}$  is hyponormal and $\sigma(|T_1|) = \{\beta_j\}_{j=1}^{\infty} \cup \{\alpha\}$ with $\|T|_{H^\bot_0}\| = \alpha$. Therefore $T_1 \in \mathcal{AM}(H^\perp_0)$ by \cite[Theorem 3.10]{BALA1}. Hence by Lemma \ref{hyponormal1}, we get

\[T_1 =
\begin{blockarray}{ccc}
H_1 & H_2 \\
\begin{block}{(cc)c}
\alpha V_1 & A & H_ 1 \\
0 & B & H_2 \\
\end{block}
\end{blockarray},
\]
where $V_1, A$ and $B$ satisfies conditions $(ii)$ and $(iii)$.

Therefore we get
\begin{equation*}
T =
\begin{blockarray}{cccc}
H_0 & H_1 & H_2 \\
\begin{block}{(ccc)c}
V_0 & 0 & 0 & H_0 \\
0 & \alpha V_1 & A & H_1 \\
0 & 0 & B & H_2\\
\end{block}
\end{blockarray}.
\end{equation*}
\end{proof}
\begin{corollary} \label{directsumANandAM}
Let $T \in \overline{\mathcal{AN}(H)}$ be hyponormal. Then $T$ can be written as $T = T_1 \oplus T_2$ on $H = H_0 \oplus H^\bot_0$ such that $T_1 \in \mathcal{AN}(H_0)$ is normal and $T_2 \in \mathcal{AM}(H^\bot_0)$ is hyponormal.
\end{corollary}
\begin{proof}
	From the representation of $T$ as in \eqref{Thypo}, it is clear that $T_1 = V_0 \in \mathcal{AN}(H_0)$ is normal and
	$T_2 =
\begin{pmatrix}
\alpha V_1 & A  \\
0 & B  \\
\end{pmatrix} \in \mathcal{AM}(H^\bot_0)$
 is hyponormal.
\end{proof}

Next we give an example of a hyponormal operator in $\overline{\mathcal{AN}(H)}$ to illustrate Theorem \ref{hypo2}.

\begin{example}
	Let $T : l^2(\mathbb{N}) \to l^2(\mathbb{N})$ be defined by
	\begin{equation*}
	T(x_1, x_2, x_3, \dots) = \left(\sqrt{\left(1-\frac{1}{2}\right)}x_2, 0, x_1, 0, x_3, \sqrt{\left(1-\frac{1}{4}\right)}x_4, x_5,  \sqrt{\left(1-\frac{1}{6}\right)}x_6, \dots\right),\ \forall \ (x_n) \in l^2(\mathbb{N}) .
	\end{equation*}
	Then
	\begin{equation*}
	T^*(x_1, x_2, x_3, \dots) = \left(x_3, \sqrt{\left(1-\frac{1}{2}\right)}x_1, x_5, \sqrt{\left(1-\frac{1}{4}\right)}x_6, x_7, \sqrt{\left(1-\frac{1}{6}\right)}x_8, \dots\right) ,\ \forall \ (x_n) \in l^2(\mathbb{N}).
	\end{equation*}
	Clearly $\|T^*x\| \leq \|Tx\|$, $\forall \ x  \in l^2(\mathbb{N})$. Hence $T$ is hyponormal. Also,
	$$T^*T(x_1, x_2, x_3, \dots) = \left(x_1,\left(1-\frac{1}{2}\right)x_2,x_3,\left(1-\frac{1}{4}\right)x_4, x_5, \left(1-\frac{1}{6}\right)x_6,\dots\right) ,\ \forall \ (x_n) \in l^2(\mathbb{N}). $$
	Therefore $T^*T = I-K \in  \overline{\mathcal{AN}(H)}$, where $$K(x_1, x_2, x_3, \dots) = \left(0,\frac{x_2}{2},0,\frac{x_4}{4},0,\frac{x_6}{6},  \dots\right), \forall \ (x_n) \in l^2(\mathbb{N}).$$
	
Hence by \cite[Theorem 3.14]{RAMSSS}, $T \in \overline{\mathcal{AN}(H)}$.

	Let $H_1 := \overline{\spn}\{e_1, e_3, e_5, \dots\}$ and $H_2 := \overline{\spn}\{e_2,e_4, e_6 \dots\}$. Then
	\[T=
	\begin{blockarray}{ccc}
	H_1& H_2\\
	\begin{block}{(c|c)c}
	V_1&A& H_1\\
	\hhline{--}0&B&H_2\\
	\end{block}
	\end{blockarray},
	\]
	where $V_1(x_1,x_2,\ldots)=(0,x_1,x_2,\dots), \forall \ (x_n) \in H_1$, \\$B(y_1,y_2,\dots)=\left(0, 0,\sqrt{\left(1-\frac{1}{4}\right)}y_2, \sqrt{\left(1-\frac{1}{6}\right)}y_3 \dots\right), \forall \ (y_n) \in H_2$ and
	
	\noindent
	 $A(y_1,y_2, \dots) = \left(\sqrt{\left(1-\frac{1}{2}\right)}y_1, 0 , \dots\right), \forall \ (y_n) \in H_2$.
	 
\noindent
Now, $BB^*(y_1, y_2,y_3\dots)	= (0,0,\left(1-\frac{1}{4}\right)y_3,\left(1-\frac{1}{6}\right)y_4,\dots), \forall \ (y_n) \in H_2$.

\noindent
 Hence $BB^* \leq \displaystyle \bigoplus_{n \in \mathbb{N}}\left(1-\frac{1}{2n}\right) I_{N\left(|T|-\sqrt{\left(1-\frac{1}{2n}\right)}\right)}$.
\end{example}

\begin{corollary}
	Let $T \in \overline{\mathcal{AN}(H)}$ be such that $T^*$ is hyponormal. Then there exists subspaces $H_0, H_1$ and $H_2$ such that $H = H_0 \oplus H_1 \oplus H_2$, with respect to this decomposition, $T$ can be written as
	
	\begin{equation}
	T =
	\begin{blockarray}{cccc}
	H_0 & H_1 & H_2 \\
	\begin{block}{(ccc)c}
	S_0 & 0 & 0 & H_0 \\
	0 & \alpha S_1 & 0 & H_1 \\
	0 & A_1 & B_1 & H_2 \\
	\end{block}
	\end{blockarray},
	\end{equation}
\end{corollary}	

where $\sigma_{ess}(|T^*|) = \{\alpha\}$,  $\alpha \geq 0$ and
 \begin{enumerate}
	\item[(i)] if $(\alpha, \|T\|] \cap \sigma(|T^*|) = \{\alpha_i\}_{i=1}^n$ for some $n \in \mathbb{N} \cup \{\infty\}$, then $S_0 = \displaystyle \bigoplus_{i=1}^{n} \alpha_i U_i$, where $U_i$ is a unitary on $N(|T^*| - \alpha_i I)$.
	
	If $\alpha = \|T\|$, then $H_0 = \{0\}$.
	
	\item[(ii)] If $H_1 \neq \{0\}$, then $S_1$ is a co-isometry on $H_1$. Otherwise $S_1 =0$.
	\item [(iii)] $A_1 \in \mathcal{B}(H_1, H_2)$ and $B_1 \in \mathcal{B}(H_2)$ are  such that $S_1A^*_1 = 0, A_1A^*_1 +B_1B^*_1= \displaystyle \bigoplus_{j=1}^{m} \beta^2_jI_j$ and $B^*_1B_1 \leq \displaystyle \bigoplus_{j=1}^{m} \beta^2_jI_j$, where $[m(T^*), \alpha) \cap \sigma(|T^*|) = \{\beta_j\}_{j=1}^m$ for some $m \in \mathbb{N} \cup \{\infty\}$ and $I_j$ is identity on $N(|T^*|- \beta_jI)$.
	
	If $\alpha = m(T^*)$, then $H_2 = \{0\}$ and $ T = \|T\|S_1$ if $\|T\|= \alpha = m(T^*)$.
\end{enumerate}

\begin{proof}
	Since $T \in \overline{\mathcal{AN}(H)}$, by \cite[Theorem 3.14]{RAMSSS}, we have  $T^*T \in \overline{\mathcal{AN}(H)}_{+}$ and hence by \cite[Theorem 4.5]{RAMSSS}, $\sigma_{ess}(T^*T) = \{\alpha\}$ for some $\alpha \geq 0$.
	
	 We first prove that $\sigma_{ess}(T^*T) = \sigma_{ess}(TT^*)$. By \cite[Theorem 6, Page 173]{MUL}, we have $\sigma_{ess}(T^*T) \setminus \{0\} = \sigma_{ess}(TT^*)\setminus \{0\}$.
	
	 If $\alpha =0$, then $T^*T \in \mathcal{K}(H)$, so is $TT^*$. Hence in this case $\sigma_{ess}(T^*T) = \sigma_{ess}(TT^*)$.
	
	Let $\alpha \neq 0$. Suppose $0 \in \sigma_{ess}(TT^*)$. Then by \cite[Theorem VII.11, Page 236]{ReedSimon} either  $0$ is an eigenvalue with infinite multiplicity or it is the limit point of $\sigma(TT^*)$.
	
	If $N(T^*)$  is infinite dimensional, then $N(T)$ is also infinite dimensional as $T^*$ is hyponormal it satisfies $N(T^*) \subset N(T)$, so $\alpha =0$ which is a contradiction.
	
	If $0$ is the limit point of $\sigma(TT^*)$, then we get $\alpha =0$ as $\sigma(TT^*) \setminus \{0\} = \sigma(T^*T) \setminus \{0\}$, contradicting our assumption. Hence $\sigma_{ess}(T^*T) = \sigma_{ess}(TT^*)$.
	
	Therefore by \cite[Theorem 4.12]{RAMSSS}, we get $T^* \in \overline{\mathcal{AN}(H)}$. By Theorem \ref{hypo2}, we have
	
	\begin{equation*}
	T^* =
	\begin{blockarray}{cccc}
	H_0 & H_1 & H_2 \\
	\begin{block}{(ccc)c}
	V_0 & 0 & 0 & H_0 \\
	0 & \alpha V_1 & A & H_1 \\
	0 & 0 & B & H_2\\
	\end{block}
	\end{blockarray}.
	\end{equation*}
	Hence,
		\begin{equation*}
	T =
	\begin{blockarray}{cccc}
	H_0 & H_1 & H_2 \\
	\begin{block}{(ccc)c}
	S_0 & 0 & 0 & H_0 \\
	0 & \alpha S_1 & 0 & H_1 \\
	0 & A_1 & B_1 & H_2 \\
	\end{block}
	\end{blockarray},
	\end{equation*}
	
where $S_0 = V^*_0, S_1 = V^*_1, A_1 = A^*,B_1 = B^* $ satisfies the required conditions from Theorem \ref{hypo2}.

\end{proof}
\begin{corollary}\label{TnormaliffBnormalV_1unitary}
Let $T$ has the representation as in Equation \eqref{Thypo} and satisfies the hypothesis of Theorem \ref{hypo2}, then $T$ is normal iff $V_1$ is unitary for $V_1 \neq 0$ and $B$ is normal.
\end{corollary}
\begin{proof}
	Let us assume that $V_1$ is unitary and $B$ is normal. Clearly it can be seen that as $V^*_1A =0$, this implies $A = 0$. Hence $T$ is normal.
	
	Conversely, let $T$ be normal. If $H_1 = N(|T|-\alpha I) \neq \{0\}$, then by Lemma \ref{quasinormalreducingsubspace}, we get $H_1$ reduces $T$. Hence $A = 0$. The normality of $T$ forces $B$ to be normal and $V_1$ to be unitary.
	
	If $H_1 =\{0\}$, then clearly $V_1 = 0$, $A=0$. In this case also, if $T$ is normal, then $B$ is normal.
\end{proof}

\section{Conditions implying the normality of hyponormal operators in $\overline{\mathcal{AN}(H)}$}

In this section, we check for normality of hyponormal operators in $\overline{\mathcal{AN}(H)}$ under some conditions. Similar kinds of results are proved for the hyponormal operators in $\mathcal{AN}(H)$ in \cite{BALA2}.

\begin{theorem}\label{invnormal1}
	Let $T \in \overline{\mathcal{AN}(H)}$ be hyponormal. If $T$ is invertible, then $T$ is normal.
\end{theorem}
\begin{proof}
	From equation \eqref{Thypo}, $T$ can be represented as,
		\begin{equation*}
	T =
	\begin{blockarray}{cccc}
	H_0 & H_1 & H_2 \\
	\begin{block}{(ccc)c}
	V_0 & 0 & 0 & H_0 \\
	0 & \alpha V_1 & A & H_1 \\
	0 & 0 & B & H_2 \\
	\end{block}
	\end{blockarray},
	\end{equation*}
 where $V_0 = \displaystyle \bigoplus_{i=1}^{n} \alpha_i U_i$ such that $U_i$ is unitary on $N(|T| - \alpha_i I)$, if  $(\alpha, \|T\|] \cap \sigma(|T|) = \{\alpha_i\}_{i=1}^n$; for some $n \in \mathbb{N} \cup \{\infty\}$. Hence $V_0$ is normal. Let
  \[T_1 =
 \begin{blockarray}{cc}
\begin{block}{(cc)}
 \alpha V_1 & A  \\
 0 & B  \\
 \end{block}
 \end{blockarray}.
 \]
Then $T_1 \in \overline{\mathcal{AN}(H_0^\bot)}$ is hyponormal and invertible. But by Corollary \ref{directsumANandAM}, $T_1 \in \mathcal{AM}(H_0^\bot)$. Therefore by \cite[Theorem 3.8]{BALA1}, we get $T_1^{-1} \in \mathcal{AN}(H)$ and hyponormal. So by \cite[Theorem 3.14]{NB *-para}, we conclude that $T_1^{-1}$ is normal, in turn $T_1$ is normal. Hence $T$ is normal.

\end{proof}

\begin{corollary}\label{invnormal2}
	Let $T \in \overline{\mathcal{AN}(H)}$ be hyponormal. If $N(T) = N(T^*)$, then $T$ is normal.
\end{corollary}
\begin{proof}
	If $T$ is compact, then by \cite[Corollary 1]{STA}, we get that $T$ is normal.
	
	Hence we assume that $T$ is not compact. Then from \cite[Corollary 3.8]{RAMSSS}, we have $N(T)$ is finite dimensional and $R(T)$ is closed. Since $T$ is hyponormal, it is clear that $N(T)$ reduces $T$ \cite[Lemma 2]{STA}. Hence with respect to $H = N(T) \oplus N(T)^\bot$, $T$ can be written as:
	\[	T =
	\begin{blockarray}{ccc}
	N(T) & N(T)^\bot \\
	\begin{block}{(cc)c}
	0 & 0 & N(T) \\
	0 & T_1 & N(T)^\bot \\
\end{block}
	\end{blockarray}.\]
	
As $N(T) = N(T^*)$, we have  $T_1 \in \overline{\mathcal{AN}(N(T)^\bot)}$ is invertible, hyponormal operator. Therefore by Theorem \ref{invnormal1}, we get $T_1$ is normal. Hence $T$ is normal.
\end{proof}
\begin{remark}
An operator $T \in \mathcal{B}(H)$ is said to be an EP-operator if $R(T)$ is closed and $R(T) = R(T^*)$. This is equivalent to saying that $N(T) = N(T^*)$ and $R(T)$ is closed. Since in Corollary \ref{invnormal2} we are considering $N(T) = N(T^*)$, this class of operators contains $EP$-operators. More details about this class can be found in \cite{CAM}.
\end{remark}

\begin{theorem}
	If $T \in \overline{\mathcal{AN}(H)}$ is hyponormal such that $\sigma_{ess}(T) = w(T)$, then $T$ is normal.
\end{theorem}
\begin{proof}
 Since $T$ is hyponormal, let $T = V|T|$ be a factorization of $T,$ such that $V$ is an isometry \cite[Page 4]{XIA}. Then by \cite[Theorem 3.7]{RAMSSS}, we get
$T = \alpha V + K_1$, where $K_1 \in \mathcal{K}(H)$. Now, from the Weyl's Theorem for the essential spectrum, we have $\sigma_{ess}(T) = \sigma_{ess}(\alpha V)$ and $\sigma_{ess}(\alpha V) \subseteq C(0, \alpha)$. As $T$ is hyponormal, by \cite[Theorem 3.1]{COB}, we have $\sigma(T) = w(T) \cup \pi_{00}(T)$. So
 \begin{equation*}
 \text{Area}(\sigma(T)) = \text{Area}(w(T)) = \text{Area}(\sigma_{ess}(T)) = \text{Area}(\sigma_{ess}(\alpha V)) = 0.
 \end{equation*}
 Hence by the Putnam's inequality \cite[Theorem 1]{Putnam}, we have
 \begin{equation*}
 \|T^*T - TT^*\| \leq \frac{1}{\pi}(\text{Area}(\sigma(T))) = 0.
\end{equation*}
This implies $T$ is normal.
 \end{proof}

There are still larger classes of non-normal operators such as paranormal operators, $\ast$-paranormal operators etc., which are defined as follows.
\begin{definition}
	Let $T \in \mathcal{B}(H)$. Then $T$ is said to be
	\begin{enumerate}
		\item paranormal if $\|Tx\|^2 \leq \|T^2x\| \|x\|$ for all $x \in H$.
		\item $\ast$- paranormal if $\|T^*x\|^2 \leq \|T^2x\| \|x\|$ for all $x \in H$.
	\end{enumerate}
	
\end{definition}

These classes of operators contain the class of hyponormal operators. Hence we end this section with the following problem.

\begin{prob}
Give a representation of $T \in \overline{\mathcal{AN}(H)}$, when $T$ is paranormal and $\ast$-paranormal. 
\end{prob}

The details of the above question(s) will be published in the forthcoming paper(s).
	\section*{Funding}
The first author is supported by SERB Grant No. MTR/2019/001307, Govt. of India. The second author is supported by the  Department of Science and Technology- INSPIRE Fellowship (Grant No. DST/INSPIRE FELLOWSHIP/2018/IF180107).

\section*{Data Availability}
Data sharing is not applicable to this article as no datasets were generated or
analyzed during the current study.

\section*{Conflict of interest}
The authors declare that there are no conflicts of interests.

\end{document}